\documentclass[11pt]{article}

\usepackage[a4paper,margin=28mm]{geometry}
\usepackage{amsmath,amssymb,amsthm,mathtools}
\usepackage{booktabs}
\usepackage{array}
\usepackage{microtype}
\usepackage[hidelinks]{hyperref}
\usepackage[nameinlink,noabbrev]{cleveref}

\newtheorem{theorem}{Theorem}[section]
\newtheorem{lemma}[theorem]{Lemma}
\newtheorem{proposition}[theorem]{Proposition}

\theoremstyle{definition}

\newtheorem{remark}[theorem]{Remark}

\newcommand{\supp}{\operatorname{supp}}

\title{The Remaining $K_4$-Free Case in the Multipartite Clique Problem}
\author{Yuuki Kasugai\\{\small Japan}}
\date{}

\begin{document}
\maketitle

\begin{abstract}
For integers $n,r,t$ with $2\le t\le r-1$, let $f(n,r,t+1)$ denote the largest possible minimum degree of a balanced $r$-partite graph with parts of size $n$ and containing no copy of $K_{t+1}$. Lo, Treglown and Zhao identified $f(n,7,4)$ as the only remaining case in their treatment of the $K_4$-free family. I determine this function for every $n\ge1$. First, the corresponding three-colourable extremum $\delta(n,7,3)$ is reduced to a $7\times3$ integer matrix problem and determined exactly. Second, a structural argument shows that every balanced $7$-partite $K_4$-free graph $G$ with
\[
\delta(G)>\frac{132}{31}n
\]
is three-colourable. Consequently,
$f(n,7,4)=\lfloor30n/7\rfloor$ except when $n\equiv4\pmod7$ and $n\ge11$, where it is one less.
\end{abstract}



\section{Introduction}

In 1975, Bollob\'as, Erd\H{o}s and Szemer\'edi asked for the maximum possible minimum degree in a balanced multipartite graph avoiding a complete graph of prescribed order~\cite{BES1975}. Following Lo, Treglown and Zhao~\cite{LTZ2022}, for integers $n,r,t$ with $2\le t\le r-1$ define
\[
f(n,r,t+1):=\max\{\delta(G):G\text{ is balanced $r$-partite with parts of size $n$ and }K_{t+1}\nsubseteq G\}.
\]
They also define
\[
\delta(n,r,t):=\max\{\delta(G):G\text{ is balanced $r$-partite with parts of size $n$ and }\chi(G)\le t\}.
\]
Every $t$-colourable graph is $K_{t+1}$-free, so
\begin{equation}\label{eq:trivial-lower}
\delta(n,r,t)\le f(n,r,t+1).
\end{equation}

Lo, Treglown and Zhao essentially determined $f(n,r,4)$ for every $r\ne7$ and explicitly singled out $f(n,7,4)$ as the only remaining case for the $K_4$-free family~\cite[Section~7]{LTZ2022}. Their Corollary~1.4 excludes $r=7$, but its $r\equiv1\pmod3$ formula has the same linear scale that would become $30n/7$ at $r=7$. The result below shows exactly what happens in this excluded case: the linear term persists, while the integral correction depends on $n\bmod7$ and is nonzero precisely when $n\equiv4\pmod7$ and $n\ge11$. It also establishes $f(n,7,4)=\delta(n,7,3)$, contributing to the equality question raised in~\cite[Section~7]{LTZ2022}.

The matrix encoding and the support inequalities used below follow the same general framework as the treatment of the colourable extremum in~\cite{LTZ2022}. The additional points needed for $r=7$ are the equality analysis in the residue class $n\equiv4\pmod7$, which shows that the general upper bound is missed by exactly one for $n\ge11$, and the structural argument in \cref{sec:structural}, which uses the original seven-part partition to exclude the non-three-colourable branch above the required threshold. The main result is the following.

\begin{theorem}\label{thm:main}
For every integer $n\ge1$,
\[
f(n,7,4)=
\begin{cases}
\left\lfloor\dfrac{30n}{7}\right\rfloor-1,& n\equiv4\pmod7\text{ and }n\ge11,\\[4mm]
\left\lfloor\dfrac{30n}{7}\right\rfloor,&\text{otherwise}.
\end{cases}
\]
Moreover,
\[
f(n,7,4)=\delta(n,7,3)
\]
for every $n\ge1$.
\end{theorem}

The proof has two independent components. First, \cref{sec:three-colour} determines $\delta(n,7,3)$ exactly. Second, \cref{sec:structural} proves that a balanced $7$-partite $K_4$-free graph whose minimum degree is only slightly above this value must already be three-colourable. The proof of \cref{thm:main} is then completed in \cref{sec:main-proof}.

\section{The three-colourable extremum}\label{sec:three-colour}

Let $G$ be a balanced $7$-partite graph with original parts
\[
V_1,\ldots,V_7,\qquad |V_i|=n,
\]
and suppose $\chi(G)\le3$. Fix a proper colouring with colour classes $W_1,W_2,W_3$ and write
\[
a_{ij}:=|V_i\cap W_j|,
\qquad
c_j:=|W_j|=\sum_{i=1}^7a_{ij}.
\]
Thus
\begin{equation}\label{eq:row-sum}
a_{i1}+a_{i2}+a_{i3}=n\qquad (i\in[7]).
\end{equation}

Every edge whose endpoints lie in distinct original parts and in distinct colour classes may be added. This operation preserves both the original $7$-partite structure and three-colourability and cannot decrease minimum degree. Consequently an extremal graph is completely described by the matrix $A=(a_{ij})$.

For a vertex in a nonempty cell $V_i\cap W_j$, the only vertices outside its own original part to which it cannot be adjacent are the $c_j-a_{ij}$ vertices of colour $j$ lying in other original parts. Hence its degree is
\[
6n-(c_j-a_{ij}).
\]
Define
\begin{equation}\label{eq:D-def}
D(A):=\max_{a_{ij}>0}(c_j-a_{ij}).
\end{equation}
Then
\begin{equation}\label{eq:delta-matrix}
\delta(n,7,3)=6n-\min_A D(A),
\end{equation}
where the minimum ranges over nonnegative integer $7\times3$ matrices satisfying \eqref{eq:row-sum}.

Lo, Treglown and Zhao proved the following general bound~\cite[Proposition~5.1]{LTZ2022}: if $m,t\ge2$ and
\[
(m-1)t<r<mt,
\]
then
\[
\delta(n,r,t)\le (r-1)n-
\left\lceil\frac{(m-1)(r-1)n}{mt-2}\right\rceil.
\]
For $r=7$ and $t=3$, the inequality $6<7<9$ gives $m=3$, and hence
\begin{equation}\label{eq:LTZ-upper}
\delta(n,7,3)\le6n-\left\lceil\frac{12n}{7}\right\rceil.
\end{equation}
Equivalently,
\begin{equation}\label{eq:D-lower}
D(A)\ge\left\lceil\frac{12n}{7}\right\rceil
\end{equation}
for every admissible matrix $A$. The following argument is the $r=7$, $t=3$ specialization of the proof of Proposition~5.1 in~\cite{LTZ2022}, included here for completeness.

\begin{proposition}\label{prop:matrix-lower}
Every nonnegative integer $7\times3$ matrix satisfying \eqref{eq:row-sum} obeys \eqref{eq:D-lower}.
\end{proposition}

\begin{proof}
Suppose instead that $D(A)<12n/7$, and write $c_j$ for the column sums. For each colour $j$, define
\[
\supp(j):=\{i\in[7]:a_{ij}>0\}.
\]
For a fixed colour $j$, let $k_j=|\supp(j)|$. If $k_j\le2$, then $c_j\le2n<18n/7$. If $k_j\ge3$, the smallest positive entry in column $j$ is at most $c_j/k_j$, so
\[
D(A)\ge c_j-\frac{c_j}{k_j}\ge\frac23c_j,
\]
and again $c_j<18n/7$.

If every row is monochromatic, then some colour occurs on at least three rows; for that colour, $D(A)\ge2n>12n/7$, a contradiction. Hence some row $i$ uses $s\in\{2,3\}$ colours. Let $C(i)$ be the set of colours present in that row. For $j\in C(i)$,
\[
c_j\le D(A)+a_{ij}<\frac{12n}{7}+a_{ij},
\]
while every $j\notin C(i)$ satisfies $c_j<18n/7$. Therefore
\[
7n=\sum_{j=1}^3c_j
< n+s\frac{12n}{7}+(3-s)\frac{18n}{7}.
\]
For $s=2$ the right-hand side equals $7n$, and for $s=3$ it is smaller than $7n$, a contradiction in either case. Thus $D(A)\ge12n/7$, and integrality gives \eqref{eq:D-lower}.
\end{proof}

\subsection{Explicit constructions}

Write
\[
n=7q+s,\qquad 0\le s\le6.
\]
The construction below comes from first taking $x_1=x_2=x_3=x$ over the reals. The two dominant defect terms are then $n+5x$ and $2n-2x$; balancing them gives $x=n/7$. The table chooses nearby integers according to the residue class of $n$.

Choose integers $x_1,x_2,x_3$ according to the following table.
\begin{center}
\begin{tabular}{c|c}
\toprule
$s$ & $(x_1,x_2,x_3)$\\
\midrule
$0,1$ & $(q,q,q)$\\
$2,3$ & $(q,q,q+1)$\\
$4,5$ & $(q,q+1,q+1)$\\
$6$   & $(q+1,q+1,q+1)$\\
\bottomrule
\end{tabular}
\end{center}
Consider the matrix
\begin{equation}\label{eq:construction-matrix}
A=
\begin{pmatrix}
n&0&0\\
x_1&n-x_1&0\\
x_2&n-x_2&0\\
x_3&n-x_3&0\\
x_1&0&n-x_1\\
x_2&0&n-x_2\\
x_3&0&n-x_3
\end{pmatrix}.
\end{equation}
Put $S=x_1+x_2+x_3$. Its column sums are
\[
c_1=n+2S,
\qquad
c_2=c_3=3n-S.
\]
A direct substitution into \eqref{eq:D-def} yields the next statement.

\begin{proposition}\label{prop:construction}
For the matrix in \eqref{eq:construction-matrix},
\[
D(A)=
\begin{cases}
\left\lceil\dfrac{12n}{7}\right\rceil+1,&n\equiv4\pmod7\text{ and }n\ge11,\\[4mm]
\left\lceil\dfrac{12n}{7}\right\rceil,&\text{otherwise}.
\end{cases}
\]
In particular,
\begin{equation}\label{eq:construction-lower}
\delta(n,7,3)\ge
\begin{cases}
\left\lfloor\dfrac{30n}{7}\right\rfloor-1,&n\equiv4\pmod7\text{ and }n\ge11,\\[4mm]
\left\lfloor\dfrac{30n}{7}\right\rfloor,&\text{otherwise}.
\end{cases}
\end{equation}
\end{proposition}

\begin{proof}
The chosen values satisfy $0\le x_i\le n$, and every row of \eqref{eq:construction-matrix} has sum $n$, so the matrix is admissible and represents a three-colourable construction. For $q\ge1$ all the relevant cells below are nonempty. The only candidate values in \eqref{eq:D-def} are
\[
2S,\qquad n+2S-x_i,\qquad 2n-S+x_i.
\]
Thus the second family is maximized at $\min_i x_i$ and the third at $\max_i x_i$. Direct substitution gives
\begin{center}
\begin{tabular}{c|c|c|c|c|c}
\toprule
$s$ & $2S$ & $n+2S-\min x_i$ & $2n-S+\max x_i$ & $D(A)$ & $\lceil12n/7\rceil$\\
\midrule
$0$ & $6q$   & $12q$   & $12q$   & $12q$   & $12q$\\
$1$ & $6q$   & $12q+1$ & $12q+2$ & $12q+2$ & $12q+2$\\
$2$ & $6q+2$ & $12q+4$ & $12q+4$ & $12q+4$ & $12q+4$\\
$3$ & $6q+2$ & $12q+5$ & $12q+6$ & $12q+6$ & $12q+6$\\
$4$ & $6q+4$ & $12q+8$ & $12q+7$ & $12q+8$ & $12q+7$\\
$5$ & $6q+4$ & $12q+9$ & $12q+9$ & $12q+9$ & $12q+9$\\
$6$ & $6q+6$ & $12q+11$& $12q+10$& $12q+11$& $12q+11$\\
\bottomrule
\end{tabular}
\end{center}
This proves the assertion for $q\ge1$. When $q=0$, some cells disappear and therefore must not contribute to the maximum in \eqref{eq:D-def}. Directly from the nonempty cells of \eqref{eq:construction-matrix}, for $n=1,\ldots,6$ one obtains
\[
D(A)=2,4,6,7,9,11,
\]
respectively, exactly the values $\lceil12n/7\rceil$. In particular, $n=4$ does not belong to the exceptional family.
\end{proof}

The exceptional residue class $n\equiv4\pmod7$, with $n\ge11$, is treated next; the one-unit deficit is unavoidable.

\subsection{The residue class \texorpdfstring{$n\equiv4\pmod7$}{n = 4 (mod 7)}}

\begin{lemma}\label{lem:residue-four}
Let $n=7q+4$ with $q\ge1$. Every admissible matrix $A$ satisfies
\[
D(A)\ge12q+8.
\]
\end{lemma}

\begin{proof}
Suppose for a contradiction that
\begin{equation}\label{eq:D-assume}
D(A)\le D_0:=12q+7.
\end{equation}
Write $c_1,c_2,c_3$ for the column sums.

First, no row can use all three colours. Indeed, if $a_{i1},a_{i2},a_{i3}>0$, then the definition of $D(A)$ gives
\[
a_{ij}\ge c_j-D_0\qquad(j=1,2,3).
\]
Using \eqref{eq:row-sum} and $c_1+c_2+c_3=7n$,
\[
n\ge7n-3D_0,
\]
so $D_0\ge2n=14q+8$, contradicting \eqref{eq:D-assume}.

Second, not every row can be monochromatic. If all seven rows were monochromatic, one colour would occur on at least three rows, and a vertex in one of those rows would have colour-defect at least $2n>D_0$. Hence at least one row is split between exactly two colours.

Fix any split row $i$, let its two colours be $p$ and $q'$, and let $\ell$ be the omitted colour. Then
\[
c_p-a_{ip}\le D_0,\qquad c_{q'}-a_{iq'}\le D_0,
\qquad a_{ip}+a_{iq'}=n,
\]
so
\begin{equation}\label{eq:omitted-lower}
c_\ell=7n-c_p-c_{q'}\ge6n-2D_0=18q+10.
\end{equation}
Let $k=|\supp(\ell)|$. If $k\le2$, then $c_\ell\le2n<18q+10$, impossible. If $k\ge3$, every positive entry in column $\ell$ is at least $c_\ell-D_0$, and hence
\[
c_\ell\ge k(c_\ell-D_0),
\qquad
c_\ell\le\frac{k}{k-1}D_0.
\]
If $k\ge4$, then
\[
c_\ell\le\frac43D_0=16q+\frac{28}{3}<18q+10,
\]
again contradicting \eqref{eq:omitted-lower}. Thus $k=3$, and integrality gives
\[
c_\ell\le\left\lfloor\frac32D_0\right\rfloor=18q+10.
\]
Consequently every split row has the following label-independent property:
\begin{equation}\label{eq:split-generic}
c_\ell=18q+10,\qquad |\supp(\ell)|=3,
\end{equation}
and equality holds in the two defect inequalities on that row,
\begin{equation}\label{eq:split-row-equality}
c_p-a_{ip}=c_{q'}-a_{iq'}=D_0.
\end{equation}

Classify split rows by the colour-pair types $12$, $13$, and $23$. It is impossible that only one split type occurs. Indeed, if only type $12$ occurred, then colour $3$ would occur only on monochromatic rows; by \eqref{eq:split-generic} its support would consist of exactly three such rows, forcing $c_3=3n$, contrary to $c_3=18q+10$. It is also impossible that all three split types occur, since the omitted colour in each type would force
\[
c_1=c_2=c_3=18q+10,
\]
whereas $3(18q+10)\ne7n$.

Therefore exactly two split types occur. Any two distinct pairs among $\{12,13,23\}$ share one colour, so after a single relabelling they may be taken to be $12$ and $13$. A $12$-split omits colour $3$, while a $13$-split omits colour $2$; hence \eqref{eq:split-generic} gives
\[
c_2=c_3=18q+10,
\qquad
c_1=13q+8.
\]
Moreover $|\supp(2)|=3$. Every non-monochromatic occurrence of colour $2$ lies in a $12$-split row, and by \eqref{eq:split-row-equality} its cell size is
\[
c_2-D_0=(18q+10)-(12q+7)=6q+3.
\]
Let $b$ be the number of monochromatic colour-$2$ rows. The three support rows therefore give
\[
18q+10=b(7q+4)+(3-b)(6q+3).
\]
After simplification,
\[
b(q+1)=1,
\]
which is impossible for $q\ge1$. This contradiction proves the lemma.
\end{proof}

Combining \cref{eq:LTZ-upper,prop:construction,lem:residue-four} gives the exact three-colourable extremum.

\begin{theorem}\label{thm:delta-exact}
For every $n\ge1$,
\[
\boxed{
\delta(n,7,3)=
\begin{cases}
\left\lfloor\dfrac{30n}{7}\right\rfloor-1,&n\equiv4\pmod7\text{ and }n\ge11,\\[4mm]
\left\lfloor\dfrac{30n}{7}\right\rfloor,&\text{otherwise}.
\end{cases}}
\]
\end{theorem}

\section{A three-colourability threshold for balanced \texorpdfstring{$7$-partite $K_4$-free}{7-partite K4-free} graphs}\label{sec:structural}

The proof uses three classical structural results. The first is due to Goddard and Lyle~\cite{GoddardLyle2011}.

\begin{theorem}[Goddard--Lyle]\label{thm:GL}
Let $H$ be a maximal $K_4$-free graph on $N$ vertices. If
\[
\delta(H)>\frac35N,
\]
then $H$ is the join of an independent set and a triangle-free graph $\Gamma$, where moreover $\delta(\Gamma)>|\Gamma|/3$.
\end{theorem}

\begin{theorem}[Andra\'sfai--Erd\H{o}s--S\'os~\cite{AES1974}]\label{thm:AES}
If a triangle-free graph $H$ on $N$ vertices satisfies
\[
\delta(H)>\frac25N,
\]
then $H$ is bipartite.
\end{theorem}

\begin{theorem}[H\"aggkvist~\cite{Haggkvist1982}]\label{thm:Haggkvist}
If a triangle-free graph $H$ on $N$ vertices satisfies
\[
\delta(H)>\frac38N,
\]
then $H$ is homomorphic to $C_5$.
\end{theorem}

The formulations above are finite statements; no ``sufficiently large $N$'' hypothesis is being used. The numerical scale also explains why the original seven-part structure must be exploited. For a general $K_4$-free graph on $N$ vertices (when the indicated divisibilities hold), the join of an independent set of size $3N/8$ with a balanced blow-up of $C_5$ on the remaining $5N/8$ vertices is $4$-chromatic and has minimum degree $5N/8$. When $N=7n$, this is $35n/8=4.375n$. Moreover, Goddard and Lyle~\cite{GoddardLyle2011} obtain a homomorphism to a $5$-wheel or a triangle only above
\[
\frac{8}{13}N=\frac{56}{13}n\approx4.3077n.
\]
Both values exceed the scale $30n/7\approx4.2857n$ needed for the present extremal problem. The counting below uses the seven original parts to pass below these general thresholds.

The following lemma is the main structural ingredient.

\begin{lemma}\label{lem:threshold}
Let $G$ be a balanced $7$-partite $K_4$-free graph with original parts $V_1,\ldots,V_7$, each of size $n$. If
\[
\delta(G)>\frac{132}{31}n,
\]
then
\[
\chi(G)\le3.
\]
\end{lemma}

\begin{proof}
Assume for a contradiction that $\chi(G)\ge4$. We first obtain a decomposition $V(G)=A\sqcup B$, then use a $C_5$-homomorphism of $G[B]$ to show that no original part can meet both $A$ and $B$. This will force $a:=|A|/n$ to be an integer while the degree bounds give $2<a<3$.

Put
\[
d:=\frac{\delta(G)}n,
\qquad
Q:=6-d.
\]
Thus
\begin{equation}\label{eq:d-threshold}
d>\frac{132}{31}.
\end{equation}

Let $G^*$ be a maximal $K_4$-free supergraph of $G$ on the same vertex set, maximal among all graphs on that vertex set. In particular, $G^*$ is not required to remain $7$-partite and may contain edges inside the original parts; no compatibility with the original seven-part partition is imposed or used in this extension. Since
\[
\delta(G^*)\ge\delta(G)>\frac{132}{31}n>\frac35(7n),
\]
\cref{thm:GL} gives a partition
\[
V(G)=A\sqcup B
\]
such that $A$ is independent in $G^*$ and $G^*[B]$ is triangle-free. Therefore $A$ is also independent in $G$ and $G[B]$ is triangle-free. If $G[B]$ were bipartite (including the degenerate case $B=\varnothing$), then two colours on $B$ together with a third colour on $A$ would give $\chi(G)\le3$, contrary to assumption. Thus $B\ne\varnothing$ and $G[B]$ is non-bipartite. The additional conclusion $\delta(G^*[B])>|B|/3$ from \cref{thm:GL} is not needed below.

Normalize sizes by $n$:
\[
a:=\frac{|A|}{n},
\qquad
a_i:=\frac{|A\cap V_i|}{n}.
\]
Then $|B|=(7-a)n$.

For $x\in B\cap V_i$, the number of its possible neighbours in $A$ is at most $(a-a_i)n\le an$, hence
\[
d_{G[B]}(x)\ge(d-a)n.
\]
Since $G[B]$ is triangle-free and non-bipartite, \cref{thm:AES} implies
\[
(d-a)n\le\frac25(7-a)n.
\]
Therefore
\begin{equation}\label{eq:a-lower}
a\ge\frac{5d-14}{3}>\frac{226}{93}>2.
\end{equation}
In particular $A\ne\varnothing$ and $A$ meets at least three of the original parts.

Now take $x\in A\cap V_i$. Since $A$ is independent and $G$ is $7$-partite, all neighbours of $x$ lie in
\[
B\setminus V_i,
\]
whose normalized size is
\[
(7-a)-(1-a_i)=6-a+a_i.
\]
Hence
\[
d\le6-a+a_i,
\]
or equivalently
\begin{equation}\label{eq:a-row}
a-a_i\le Q
\end{equation}
for every $i$ with $a_i>0$.
If $A$ meets $s\ge3$ original parts, \eqref{eq:a-row} gives $a_i\ge a-Q$ on each of those parts. Summing over the $s$ parts gives
\[
a\ge s(a-Q),
\]
and therefore
\begin{equation}\label{eq:a-upper}
a\le\frac{s}{s-1}Q\le\frac32Q.
\end{equation}

The function $(d-a)/(7-a)$ is decreasing in $a$ because $d<7$. Using \eqref{eq:a-upper},
\begin{align}
\frac{\delta(G[B])}{|B|}
&\ge\frac{d-a}{7-a}\notag\\
&\ge\frac{d-\frac32(6-d)}{7-\frac32(6-d)}
=\frac{5d-18}{3d-4}.\label{eq:B-ratio}
\end{align}
The inequality
\[
\frac{5d-18}{3d-4}>\frac38
\]
is equivalent to $31d>132$, which holds by \eqref{eq:d-threshold}. Thus \cref{thm:Haggkvist} gives a homomorphism
\[
\varphi:G[B]\to C_5.
\]
Because $G[B]$ is non-bipartite, the image of $\varphi$ cannot be a proper subgraph of $C_5$; every proper subgraph of $C_5$ is bipartite. Consequently all five fibres are nonempty. Write
\[
B_0,\ldots,B_4
\]
for the fibres, indexed modulo $5$, and set
\[
x_i:=\frac{|B_i|}{n},
\qquad
b_{ij}:=\frac{|B_i\cap V_j|}{n}.
\]
For $i\in\mathbb Z/5\mathbb Z$, define
\[
t_i:=x_i+x_{i+2}+x_{i+3},
\qquad
t_{ij}:=b_{ij}+b_{i+2,j}+b_{i+3,j}.
\]
Vertices in $B_i$ have no neighbours in $B_i\cup B_{i+2}\cup B_{i+3}$. Thus, if $b_{ij}>0$, then among the $6n$ vertices outside $V_j$, at least $(t_i-t_{ij})n$ lie in these three fibres and are nonadjacent to a vertex of $B_i\cap V_j$. Possible further non-neighbours in $A$ are discarded from this upper-bound calculation. Hence
\begin{equation}\label{eq:t-local}
t_i-t_{ij}\le Q.
\end{equation}
As $B_i$ is nonempty, choose $j$ with $b_{ij}>0$. Since $t_{ij}\le1$, \eqref{eq:t-local} gives
\begin{equation}\label{eq:t-upper}
t_i\le Q+1
\qquad(i\in\mathbb Z/5\mathbb Z).
\end{equation}

Suppose some original part $V_j$ meets both $A$ and $B$. Choose $i$ with $b_{ij}>0$. From \eqref{eq:a-row},
\[
a_j\ge a-Q,
\]
so the total normalized $B$-mass in $V_j$ is at most
\[
1-a_j\le1-a+Q.
\]
Hence $t_{ij}\le1-a+Q$, and \eqref{eq:t-local} yields
\begin{equation}\label{eq:t-mixed}
t_i\le1-a+2Q.
\end{equation}
Each $x_k$ occurs in exactly three of the five numbers $t_i$, so
\[
\sum_{i=0}^4t_i=3\sum_{i=0}^4x_i=3(7-a).
\]
Using \eqref{eq:t-mixed} for the selected $i$ and \eqref{eq:t-upper} for the other four indices,
\[
3(7-a)\le(1-a+2Q)+4(Q+1).
\]
Thus
\begin{equation}\label{eq:a-mixed-lower}
a\ge8-3Q.
\end{equation}
Together with \eqref{eq:a-upper}, this implies
\[
8-3Q\le\frac32Q,
\qquad
Q\ge\frac{16}{9}.
\]
But \eqref{eq:d-threshold} gives
\[
Q=6-d<6-\frac{132}{31}=\frac{54}{31}<\frac{16}{9},
\]
a contradiction. Therefore no original part mixes $A$ and $B$.

It follows that every original part meeting $A$ lies entirely in $A$. Hence $a$ is an integer. By \eqref{eq:a-lower}, $a>2$, so $a\ge3$. On the other hand, \eqref{eq:a-upper} and \eqref{eq:d-threshold} give
\[
a\le\frac32Q<\frac32\cdot\frac{54}{31}=\frac{81}{31}<3,
\]
again a contradiction. Therefore $\chi(G)\le3$.
\end{proof}

\begin{remark}[Bounds on the optimal threshold]\label{rem:threshold-range}
Let $c_*$ be the infimum of the real numbers $c$ such that, for every $n\ge1$, every balanced $7$-partite $K_4$-free graph $G$ with parts of size $n$ and $\delta(G)>cn$ is three-colourable. A balanced $7$-partite counterexample at degree $4n$ is obtained by taking two original parts with no edges between them as an independent set $A$, taking the remaining five original parts as the five fibres of a $C_5$ blow-up, and joining $A$ completely to those five parts. The resulting graph is $K_4$-free, has chromatic number $4$, and has minimum degree $4n$, so $c_*\ge4$. On the other hand, \cref{lem:threshold} shows that $132/31$ is sufficient. Therefore
\[
4\le c_*\le\frac{132}{31}.
\]
Within the present proof, $132/31$ arises exactly when the lower bound in \eqref{eq:B-ratio} crosses H\"aggkvist's $3/8$ threshold. Once a five-fibre $C_5$ homomorphism is available, the later mixed-part contradiction only requires $Q<16/9$, equivalently $d>38/9$, while the Goddard--Lyle reduction itself requires only $d>21/5$. One possible route to improvement is a lower-threshold structural theorem that still provides enough controlled fibre structure for the counting argument. The $10/29$ triangle-free three-colourability theory of Jin~\cite{Jin1993} does not by itself provide a $C_5$ homomorphism, so it cannot simply be substituted for H\"aggkvist's theorem in the present proof.
\end{remark}

\section{Proof of the main theorem}\label{sec:main-proof}

Set
\[
M(n):=\delta(n,7,3).
\]
Its exact value is given by \cref{thm:delta-exact}.

\begin{lemma}\label{lem:threshold-gap}
For every $n\ge1$,
\[
M(n)+1>\frac{132}{31}n.
\]
\end{lemma}

\begin{proof}
Write $n=7q+s$, $0\le s\le6$. If $s\ne4$, or if $n=4$, then $M(n)=\lfloor30n/7\rfloor$, and direct reduction by residues gives
\[
31(M(n)+1)-132n=6q+c_s,
\]
where
\[
(c_0,c_1,c_2,c_3,c_4,c_5,c_6)=(31,23,15,7,30,22,14).
\]
All values are positive. In the exceptional case $s=4$ and $q\ge1$, \cref{thm:delta-exact} gives $M(n)=30q+16$, so
\[
31(M(n)+1)-132n
=31(30q+17)-132(7q+4)
=6q-1>0.
\]
The smallest margin occurs at $q=1$, i.e. $n=11$, where
\[
M(11)+1-\frac{132}{31}\,11=\frac{5}{31}.
\]
\end{proof}

\begin{proof}[Proof of Theorem~\ref{thm:main}]
The lower bound
\[
f(n,7,4)\ge M(n)
\]
follows from \eqref{eq:trivial-lower}.

Suppose, for a contradiction, that
\[
f(n,7,4)\ge M(n)+1.
\]
Then there exists a balanced $7$-partite $K_4$-free graph $G$ with
\[
\delta(G)\ge M(n)+1>\frac{132}{31}n
\]
by \cref{lem:threshold-gap}. The structural \cref{lem:threshold} implies $\chi(G)\le3$. But by the definition of $M(n)=\delta(n,7,3)$, every three-colourable balanced $7$-partite graph with parts of size $n$ satisfies
\[
\delta(G)\le M(n),
\]
contradicting $\delta(G)\ge M(n)+1$. Thus
\[
f(n,7,4)\le M(n).
\]
Together with the opposite inequality, this proves
\[
f(n,7,4)=M(n),
\]
and the explicit formula follows from \cref{thm:delta-exact}.
\end{proof}

\section{Concluding remarks}

The argument separates the problem into an integer optimization inside the three-colourable class and a structural exclusion of the non-three-colourable branch. The general upper bound for $\delta(n,7,3)$ is not attained precisely when $n\equiv4\pmod7$ and $n\ge11$; the obstruction is detected by equality conditions in the support inequalities and reduces to the impossible integer relation $b(q+1)=1$.

In particular, this resolves the excluded $r=7$, $t=3$ case and gives $f(n,7,4)=\delta(n,7,3)$ for every $n\ge1$.

\section*{Declaration of Generative AI Use}

ChatGPT (OpenAI) was used to suggest potential arguments for selected parts of the proof, as well as to assist with reviewing the proof and translating the manuscript from Japanese into English. All mathematical arguments incorporated into the final manuscript were independently verified by the author, who assumes full responsibility for the content.

\end{document}